\documentclass[12pt, letterpaper]{article}
\usepackage{mathrsfs}
\usepackage{indentfirst}
\usepackage{amsmath,amssymb,amsfonts,amsthm,graphics}
\usepackage{makeidx}
\usepackage{amsmath,amssymb,amsfonts,amsthm,graphics,graphicx}
\usepackage{makeidx}
\usepackage{color}
\usepackage{float}

\newtheorem{thm}{Theorem}[section]

\newtheorem{pro}[thm]{Proposition}

\theoremstyle{definition}

\def\-{\mbox{--}}

\def\pf{\noindent {\it Proof.} }

\begin{document}
\title{Distance proper connection of graphs and their complements}
\author{
\small Xueliang Li$^1$, Colton Magnant$^2$, Meiqin Wei$^1$, Xiaoyu Zhu$^1$\\
\small $^1$Center for Combinatorics and LPMC\\
\small Nankai University, Tianjin 300071, China\\
\small Email: lxl@nankai.edu.cn; weimeiqin8912@163.com; zhuxy@mail.nankai.edu.cn\\
\small $^2$Department of Mathematical Sciences\\
\small Georgia Southern University, Statesboro, GA 30460-8093, USA\\
\small Email: cmagnant@georgiasouthern.edu}
\date{}
\maketitle

\begin{abstract}
Let $G$ be an edge-colored connected graph. A path $P$ in $G$ is called a distance $\ell$-proper path if no two edges of the same color can appear with less than $\ell$ edges in between on $P$. The graph $G$ is called $(k,\ell)$-proper connected if there is an edge-coloring such that every pair of distinct vertices of $G$ are connected by $k$ pairwise internally vertex-disjoint distance $\ell$-proper paths in $G$. The minimum number of colors needed to make $G$ $(k,\ell)$-proper connected is called the $(k,\ell)$-proper connection number of $G$ and denoted by $pc_{k,\ell}(G)$. In this paper we first focus on the $(1,2)$-proper connection number of $G$ depending on some constraints of $\overline G$. Then, we characterize the graphs of order $n$ with $(1,2)$-proper connection number $n-1$ or $n-2$. Using this result, we investigate the Nordhaus-Gaddum-Type problem of $(1,2)$-proper connection number and prove that $pc_{1,2}(G)+pc_{1,2}(\overline{G})\leq n+2$ for connected graphs $G$ and $\overline{G}$. The equality holds if and only if $G$ or $\overline{G}$ is isomorphic to a double star.
\end{abstract}

\section{Introduction}
All graphs in this paper are finite, undirected, simple and connected. We follow the notation and terminology in the book \cite{BM}.

When considering the transmission of information between agencies of the government, an immediate question is put forward as follows: What is the minimum number of passwords or firewalls needed that allows one or more secure paths between every two agencies so that the passwords along each path are distinct? This question can be represented by a graph and studied by means of what is called rainbow colorings introduced by Chartrand et al.~in \cite{CJMZ}. An \emph{edge-coloring} of a graph is a mapping from its edge set to the set of natural numbers (colors). A path in an edge-colored graph with no two edges sharing the same color is called a \emph{rainbow path}. A graph $G$ with an edge-coloring $c$ is said to be \emph{rainbow connected} if every pair of distinct vertices of $G$ is connected by at least one rainbow path in $G$. The coloring $c$ is called a \emph{rainbow coloring} of the graph $G$. For a connected graph $G$, the minimum number of colors needed to make $G$ rainbow connected is defined as the \emph{rainbow connection number} of $G$ and denoted by $rc(G)$. Many researchers have been studied problems on the rainbow connection and got plenty of nice results, see \cite{KY,LS2,LS} for examples. For more details we refer to the survey paper \cite{LSS} and the book \cite{LS}.

A relaxation of this question can be the following: What is the minimum number of passwords or firewalls that allows one or more secure paths between every two agencies such that as we progress from one agency to another along such a path, we are required to change passwords at each step? Inspired by this, Borozan et al.~in \cite{BFG} and Andrews et al.~in \cite{ALLZ} introduced the concept of proper-path coloring of graphs. Let $G$ be an edge-colored graph. A path $P$ in $G$ is called a \emph{proper path} if no two adjacent edges of $P$ are colored with the same color. An edge-colored graph $G$ is \emph{$k$-proper connected} if every pair of distinct vertices $u,v$ of $G$ are connected by $k$ pairwise internally vertex-disjoint proper $(u,v)$-paths in $G$. For a connected graph $G$, the minimum number of colors needed to make $G$ $k$-proper connected is called the \emph{$k$-proper connection number} of $G$ and denoted by $pc_k(G)$. Particularly for $k=1$, we write $pc_1(G)$, the proper connection number of $G$, as $pc(G)$ for simplicity. Recently, many results have been obtained about the proper connection number. For details we refer to a dynamic survey paper \cite{LM}.

Extending the notion of a proper path, the $(k,\ell)$-proper-path coloring was defined in \cite{SUBMITTED} as a generalization of rainbow coloring and proper-path coloring. A path $P$ in $G$ is called a \emph{distance $\ell$-proper path} if no two edges of the same color can appear with fewer than $\ell$ edges in between on $P$. The graph $G$ is called \emph{$(k,\ell)$-proper connected} if there is an edge-coloring $c$ such that every pair of distinct vertices of $G$ are connected by $k$ pairwise internally vertex-disjoint distance $\ell$-proper paths in $G$. This coloring is called a \emph{$(k,\ell)$-proper-path coloring} of $G$. In addition, if $t$ colors are used, then $c$ is referred to as a \emph{$(k,\ell)$-proper-path $t$-coloring} of $G$. For a connected graph $G$, the minimum number of colors needed to make $G$ $(k,\ell)$-proper connected is called the \emph{$(k,\ell)$-proper connection number} of $G$ and denoted by $pc_{k,\ell}(G)$. Particularly, for $k=1$ and $\ell=2$, there is an edge-coloring using $pc_{1,2}$ colors such that there exists a $2$-proper path between each pair of vertices of the graph $G$. Furthermore, if we ensure that every path in $G$ is a $2$-proper path, then the edge-coloring becomes a strong edge-coloring. In addition, the strong chromatic index $\chi'_s(G)$, which was introduced by Fouquet and Jolivet \cite{FJ}, is the minimum number of colors needed in a  strong edge-coloring of $G$. Immediately we get that $pc_{1,2}(G)\leq \chi'_s(G)$. And this inspires us to pay our attention to the $(1,2)$-proper connection number of the connected graph $G$, i.e., $pc_{1,2}(G)$.

In this paper, we consider the $(k, \ell)$-proper connection number of graphs and their complements. This paper is organized as follows. In Section~\ref{Sect:Prelim}, we list some useful results about the $(k,\ell)$-proper connection number of a graph. In Section~\ref{Sect:complements}, we focus on $pc_{1,2}(G)$ depending on some constraints of $\overline G$. In Section~\ref{Sect:NG}, we first characterize the graphs of order $n$ with $(1,2)$-proper connection number $n-1$ or $n-2$. Using this result, we give the Nordhaus-Guddum-Type result for the $(1,2)$-proper connection number, i.e., $pc_{1,2}(G)+pc_{1,2}(\overline{G})\leq n+2$ for connected graphs $G$ and $\overline{G}$, and the equality holds if and only if $G$ or $\overline{G}$ is isomorphic to a double star.

\section{Preliminaries}\label{Sect:Prelim}

In this section, we introduce some definitions and present several results which will be used later. Let $G$ be a connected graph. We denote by $n$ the number of its vertices and $m$ the number of its edges. The \emph{distance between two vertices} $u$ and $v$ in $G$, denoted by $d(u,v)$, is the length of a shortest path between them in $G$. The \emph{eccentricity} of a vertex $v$ is $ecc(v):=max_{x\in V(G)}d(v, x)$. The \emph{radius} of $G$ is $rad(G):=min_{x\in V(G)}ecc(x)$. We also write $\sigma'_2(G)$ as the largest sum of degrees of vertices $x$ and $y$, where $x$ and $y$ are taken over all couples of adjacent vertices in $G$. Additionally, we set $[n]=\{1,2,\cdots,n\}$ for any integer $n\geq1$.

The following are some results that we will use in our proofs. The first is a simple observation that the addition of edges cannot increase the proper connection number.

\begin{pro}[\cite{SUBMITTED}]\label{spanning}
If $G$ is a nontrivial connected graph and $H$ is a connected spanning subgraph of $G$, $\ell\geq 1$ is an integer. Then $pc_{1,\ell}(G)\leq pc_{1,\ell}(H)$. Particularly, $pc_{1,\ell}(G)\leq pc_{1,\ell}(T)$ for every spanning tree $T$ of $G$.
\end{pro}

When we focus on trees, the following holds.

\begin{thm}[\cite{SUBMITTED}]\label{tree}
If $T$ is a nontrivial tree, then $pc_{1,2}(T)=\sigma'_2(T)-1$.
\end{thm}

For complete bipartite graphs, the situation is trickier.

\begin{thm}[\cite{SUBMITTED}]\label{bipartite}
Let $\ell \geq 2$ be an integer and $m\leq n$. Then,
\begin{eqnarray*}
pc_{1,\ell}(K_{m,n})=\left\{
\begin{array}{rcl}
n     &      & if\ m=1,\\
2     &      & if\ m\geq2\ and\ m\leq n\leq2^m,\\
3     &      & if\ \ell=2,\ m\geq2\ and\ n>2^m\\
~     &      & or\ \ell\geq3,\ m\geq2\ and\ 2^m<n\leq3^m,\\
4     &      & if\ \ell\geq3,\ m\geq2\ and\ n>3^m.
\end{array} \right.
\end{eqnarray*}
\end{thm}

For a general $2$-connected graph, we gave in \cite{SUBMITTED} an upper bound for the $(1,2)$-proper connection number.

\begin{thm}[\cite{SUBMITTED}]\label{$2$-connected}
If a graph $G$ is $2$-connected, then $pc_{1,2}(G)\leq5$.
\end{thm}

\section{$(1,2)$-proper connection number for the complement of a graph}\label{Sect:complements}

In this section, we investigate the $(1,2)$-proper connection number of $G$ depending on some properties of its complement $\overline G$.

\begin{thm}\label{diam $4$}
If $G$ is a graph with $diam(\overline G)\geq4$, then $pc_{1,2}(G)\leq3$.
\end{thm}
\pf We first claim that $G$ must be connected. If not, $\overline G$ must contain a spanning complete bipartite graph which implies that $diam(\overline G)\leq2$, a contradiction. Choose a vertex $x$ with $ecc_{\overline G}(x)=diam(\overline G)$. Let $N_i(x)=\{v:dist_{\overline G}(x,v)=i\}$ for $0\leq i\leq3$ and  $N_4(x)=\{v:dist_{\overline G}(x,v)\geq4\}$. Obviously $N_0(x)=\{x\}$. We write $N_i$ (for $0\leq i\leq4$) instead of $N_i(x)$ and $n_i$ instead of $|N_i|$ for convenience. It can be deduced that all edges are present in $G$ of the form $uv$ where $u\in N_1$ and $v\in N_3\bigcup N_4$ or $u\in N_2$ and $v\in N_4$ (see Figure~\ref{fig1}).
\begin{figure}[H]
\begin{center}
\includegraphics[scale = 0.9]{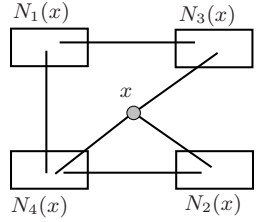}
\caption{The graph $G$ for the proof of Theorem \ref{diam $4$}.}\label{fig1}
\end{center}
\end{figure}

We denote by $N_{i,j}(0\leq i\neq j\leq4)$ the edge set between $N_i$ and $N_j$ in $G$. We distinguish four cases and give each of the cases a $(1,2)$-proper-path $3$-coloring, respectively. Again we use $f(e)(e\in E(\overline G))$ to represent the color assigned to $e$.

Case 1. If $n_4>1$. We give all edges of $N_{1,3}$ the color $3$, edges of $N_{0,3}$ the color $3$, edges of $N_{0,4}$ the color $2$, edges of $N_{0,2}$ the color $3$, edges of $N_{2,4}$ the color $1$. Additionally, color the edges of $N_{1,4}$ such that for $v\in N_1$, $\{f(vs):s\in N_4\}=\{1,2\}$. Then for any $u,v\in N_1(if~n_1>1)$, there must exist $s_1,s_2\in N_4$ (possibly with $s_{1} = s_{2}$) such that $f(us_1)=1$ and $f(vs_2)=2$. Then one of $us_{1}v$ or $us_1xss_2v$, where $s\in N_2$, is a distance $2$-proper $(u,v)$-path. Other situations can be checked similarly.

Case 2. If $n_4=1$, $n_3>1$ and $n_1=1$. Then we give all edges of $N_{1,3}$ the color $1$, the edge of $N_{1,4}$ the color $3$, edges of $N_{0,3}$ the color $1$, edges of $N_{0,4}$ the color $2$, edges of $N_{0,2}$ the color $1$ and edges of $N_{2,4}$ the color $3$. It is easy to verify this is indeed a $(1,2)$-proper-path $3$-coloring of $G$.

Case 3. If $n_4=1$, $n_3>1$ and $n_1>1$. Let $G'$ be the complete bipartite graph $G'=G[N_1\cup N_3]$. By Theorem~\ref{bipartite}, we can use at most three colors to make $G'$ $(1,2)$-proper connected. Then we give all edges of $N_{1,4}$ the color $1$, edges of $N_{0,3}$ the color $2$, the edge of $N_{0,4}$ the color $3$, edges of $N_{0,2}$ the color $1$ and edges of $N_{2,4}$ the color $2$. One can easily check this is a $(1,2)$-proper-path $3$-coloring of $G$ and we omit the details here.

Case 4. If $n_4=1$ and $n_3=1$. Then we give all edges of $N_{1,3}$ the color $1$, edges of $N_{1,4}$ the color $1$, the edge of $N_{0,3}$ the color $2$, the edge of $N_{0,4}$ the color $3$, edges of $N_{0,2}$ the color $2$ and edges of $N_{2,4}$ the color $1$. We can again verify the correctness easily.

Thus, the proof is completed.\qed

\begin{thm}\label{diam $3$}
For a graph $G$, if $\overline G$ is triangle-free and $diam(\overline G)=3$, then $pc_{1,2}(G)\leq3$.
\end{thm}
\pf As in the proof of Theorem~\ref{diam $4$}, it is easy to show that $G$ is connected. Choose a vertex $x$ such that $ecc_{\overline G}(x)=diam(\overline G)=3$. In addition, $N_i$, $n_i$ and $N_{i,j}$ for $0 \leq i \neq j \leq 3$ are defined as in the previous theorem. Again it can be deduced that there exist all edges of the form $uv$ where $u\in N_0$ and $v\in N_2\cup N_3$ or where $u\in N_1$ and $v\in N_3$. Since $\overline G$ is triangle-free and $x$ has all edges to $N_{1}$ in $\overline{G}$, we know that $N_1$ is a clique in $G$. We give a $(1,2)$-proper-path $3$-coloring for $G$ as follows.

We assign to the edges of $N_{0,2}$ the color $3$, edges of $N_{0,3}$ the color $1$, edges of $N_{1,3}$ the color $2$, any edges of $N_{1,2}$ the color $3$, any edges of $N_{2,3}$ the color $2$ and the edges of the induced subgraph $G[N_1]$ the color $3$.

It is obvious that for any $u\in N_i$ and $v\in N_j(i\neq j)$, there exists a distance $2$-proper path between them. Then it suffices to show that for any $u,v\in N_2$ or $N_3$, there is a distance $2$-proper path connecting them in $G$. First suppose $u,v\in N_2$ and there is no edge between them in $G$. Since $\overline G$ is triangle-free, there exists a vertex $w\in N_1$ such that $wv\in G$, then $uxtwv$ is a distance $2$-proper path between $u$ and $v$, where $t\in N_3$. The situation for any vertices $u,v\in N_3$ can be dealt with similarly. Thus $pc_{1,2}(G)\leq3$.\qed

\begin{thm}\label{diam $2$}
Let $G$ be a connected graph. If $\overline G$ is triangle free and $diam(\overline G)=2$, then $pc_{1,2}(G)\leq3$.
\end{thm}
\pf First we choose a vertex $x$ with $ecc_{\overline G}(x)=diam(\overline G)=2$. In addition, $N_i,~n_i$ and $N_{i,j}$ are defined as above. Clearly, all edges of the form $xv$ for $v\in N_2$ are present in $G$. Again $N_1$ is a clique in $G$ since all edges of the form $xu$ are in $\overline{G}$ for $u \in N_{1}$ and $\overline{G}$ is triangle free.

Suppose there exists a vertex $v_0\in N_2$ such that no edge $vw(w\in N_1)$ exists in $G$. Then $v_0$ is adjacent to every vertex of $N_1$ in $\overline G$. Thus, since every vertex of $N_{2}$ has at least one edge to $N_{1}$ in $\overline{G}$, the vertex $v_0$ must be adjacent to every other vertex of $N_2$ in $G$ since otherwise a triangle will appear in $\overline G$. Next we give an edge coloring $f$ for $G$. We set $f(xv_0)=3$, $f(xw)=2$ and $f(v_0w)=1~(w\in N_2, w\neq v_0)$. And we give any edges of $N_{1,2}$ the color $2$, the edges of the induced subgraph $G[N_1]$ the color $3$. We only need to consider the $2$-proper path for $w_1,w_2\in N_2$ and $w_1v_0xw_2$ clearly suffices.

Next suppose there exists no such vertex $v_0$. Since $G$ and $\overline{G}$ connected, we know that $n_1\geq2$. We denote by $E_G(v)$ (for $v\in N_2$) the set of edges between $v$ and vertices of $N_1$ in $G$ and set $e_G(v)=|E_G(v)|$. Also $e_{\overline G}(v)$ (for $v\in N_2$) is defined similarly. Again we distinguish two cases to analyze.

If $|N_1|\geq3$, for each $u\in N_2$ with $e_G(u)=1$, we give this edge the color $1$. And for $u\in N_2$ with $e_G(u)\geq2$, we arbitrarily color these edges but confirm that $\{f(e):~e\in E_G(u)\}=\{1,2\}$. Then we set $f(xu)=2~(u\in N_2)$ and give the edges of the induced subgraph $G[N_1]$ the color $3$. The rest edges are colored arbitrarily with colors from $[3]$. Again we only need to consider the distance $2$-proper path between the two non-adjacent vertices $v,w\in N_2$. Since $|N_1|\geq3$ and $v$ and $w$ are non-adjacent in $G$, so $e_{\overline G}(v)+e_{\overline G}(w)\leq|N_1|$. Thus $e_G(v)+e_G(w)\geq|N_1|\geq3$ which implies that one of the vertices $v,w$, say $v$, must have $e_G(v)\geq2$. So there exists one vertex $s\in N_1$ or two vertices $s,t\in N_1$ such that $vsw$ or $vstw$ is a distance $2$-proper $(v,w)$-path in $G$.

If $|N_1|=2$ and $N_1=\{s,t\}$. Then each vertex $u\in N_2$ is adjacent to only one vertex of $N_1$ in $G$, either $s$ or $t$ since otherwise $diam(\overline{G}) \geq 3$. We denote by $V_1$ the set of vertices of $N_2$ adjacent to $s$ in $G$, that is, the set adjacent to $t$ in $\overline G$. And we write $V_2$ for the rest of the vertices of $N_2$. It is easy to see that $V_1$ and $V_2$ both induce cliques in $G$. We then set $f(xu)~(u\in V_1)=1$, $f(us)~(u\in V_1)=2$, $f(xu)~(u\in V_2)=2$, $f(ut)~(u\in V_2)=1$, $f(st)=3$ and color any remaining edges with color $1$. It is easy to check that this is a $(1,2)$-proper-path $3$-coloring of $G$. Thus the proof is completed.\qed

\section{Nordhaus-Gaddum-Type theorem for $(1,2)$-proper connection number}\label{Sect:NG}

In this section, we first characterize the graphs on $n$ vertices with $(1,2)$-proper connection number $n-1$ or $n-2$, which is crucial to investigate the Nordhaus-Gaddum-Type result for the $(1,2)$-proper connection number of the graph $G$. We use $C_n,S_n$ to denote the cycle and the star graph on $n$ vertices, respectively. Denote by $T(n_1,n_2)$ the double star in which the degrees of its (adjacent) center vertices are $n_1+1$ and $n_2+1$ respectively. Additionally, we write $T^1(n_1,n_2)$ as the graph obtained by replacing one pendent edge with $P_3$ in the double star $T(n_1,n_2)$ and denote the new pendent vertex by $u_0$ (see Figure \ref{fig2}). Also define graphs $G_{1}, \dots, G_{8}$ as in Figure~\ref{fig2}.

\begin{figure}[H]
\begin{center}
\includegraphics[scale = 0.8]{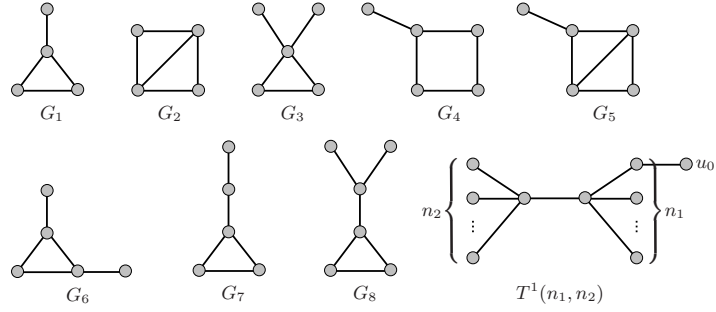}
\caption{Graphs $G_i~(1\leq i\leq 8)$ and $T^1(n_1,n_2)$ in $\mathcal{G}_2$.}\label{fig2}
\end{center}
\end{figure}

\begin{thm}\label{decrible n-1 and n-2}
Let $G$ be a nontrivial connected graph on $n\geq 2$ vertices. Then

$(\romannumeral 1)$ $pc_{1,2}(G)=n-1$ if and only if $G\in \mathcal{G}_1=\{S_n\ (n\geq 2),~T(n_1,n_2)\ (n_1,n_2\geq 1)\}$;

$(\romannumeral 2)$ $pc_{1,2}(G)=n-2$ if and only if $G\in \mathcal{G}_2=\{C_3,~C_4,~C_5,~G_1,~G_2,~G_3,~G_4,\\~G_5,~G_6,
G_7,~G_8,~T^1(n_1,n_2)\}$.
\end{thm}

\begin{proof}
Let $G$ be a connected graph of order $n\geq 2$ and $T$ be a spanning tree of $G$. Proposition~\ref{spanning} shows that $pc_{1,2}(G)\leq pc_{1,2}(T)$. Now we give proofs for $(\romannumeral 1)$ and $(\romannumeral 2)$ separately.

Proof of $(\romannumeral 1)$: For any graph $G\in \mathcal{G}_1$, we can easily check that $pc_{1,2}(G)=n-1$. So it remains to verify the converse. Since $pc_{1,2}(G)=n-1$, we see that $n-1=pc_{1,2}(G)\leq pc_{1,2}(T)\leq n-1$, i.e., $pc_{1,2}(T)=n-1$. Thus, by Theorem~\ref{tree}, we know that any spanning tree $T$ of $G$ must be a star or a double star, i.e., $T\in \mathcal{G}_{1}$.%\{S_n\ (n\geq 2),~T(n_1,n_2)\ (n_1,n_2\geq 1)\}$.
Without loss of generality, we can assume that $n_2\geq n_1$.
\begin{figure}[H]
\begin{center}
\includegraphics[scale = 0.9]{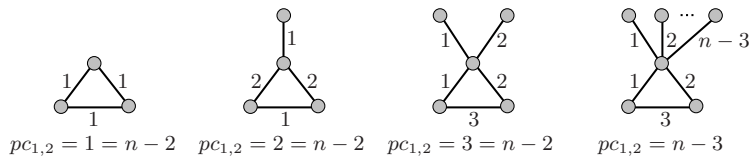}
\caption{Graphs obtained by adding an edge to  $S_n\ (n\geq 2)$.}\label{fig3}
\end{center}
\end{figure}

If $G$ is a tree, then $G\in \mathcal{G}_{1}$. %\{S_n\ (n\geq 2),~T(n_1,n_2)\ (n_1,n_2\geq 1)\}$.
Now we suppose that $G$ is not a tree. Then since $T \in \mathcal{G}_{1}$, $G$ can be constructed from $S_n\ (n\geq 2)$ or $T(n_1,n_2)\ (n_1,n_2\geq 1)$ by adding edges. Adding an edge to $S_n\ (n\geq 2)$, we will obtain one of the graphs depicted in Figure~\ref{fig3}. However, all the graphs in Figure \ref{fig3} have $(1,2)$-proper connection number no more than $n-2$, which implies that any spanning tree $T$ of $G$ cannot be a star. Next, we will consider the graphs obtained by adding edges to $T(n_1,n_2)\ (n_1,n_2\geq 1)$.

If $n_1=n_2=1$, then $T(1,1)=P_4$. If an edge is added, then we will obtain either the cycle $C_4$ or the graph $G_1$ depicted in Figure \ref{fig2}. Obviously, both $C_4$ and $G_1$ have $(1,2)$-proper connection number $2=n-2<n-1$. For the cases $n_1=1,~n_2=2$ and $n_1=n_2=2$, one of the graphs in Figure~\ref{fig4} or~\ref{fig5} will be obtained by adding an edge to $T(1,2)$ or $T(2,2)$ respectively. The $(1,2)$-proper-path colorings given in Figures~\ref{fig4} and~\ref{fig5} show that all these graphs have $(1,2)$-proper connection number no more than $n-2$.
\begin{figure}[H]
\begin{center}
\includegraphics[scale = 0.9]{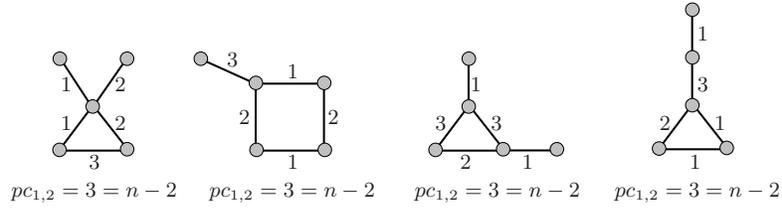}
\caption{Graphs obtained by adding an edge to $T(1,2)$.}\label{fig4}
\end{center}
\end{figure}

\begin{figure}[H]
\begin{center}
\includegraphics[scale = 0.9]{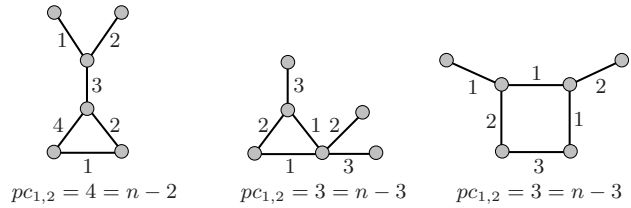}
\caption{Graphs obtained by adding an edge to $T(2,2)$.}\label{fig5}
\end{center}
\end{figure}

For all the other situations, i.e., $n_1=1,~n_2\geq 3$ or $n_1=2,~n_2\geq 3$ or $n_1\geq 3,~n_2\geq 3$, Figure \ref{fig6}, Figure \ref{fig7} and Figure \ref{fig8} give all the graphs obtained by adding an edge to $T(1,n_2\geq 3)$, $T(2,n_2\geq 3)$ and $T(n_1\geq 3,n_2\geq 3)$, respectively. We give $(1,2)$-proper-path colorings for these graphs showed in Figure \ref{fig6}, Figure \ref{fig7} and Figure \ref{fig8}. One can easily check that all these graphs have $(1,2)$-proper connection number no more than $n-2$.

From the discussions all above, we come to a conclusion that if $pc_{1, 2}(G) = n - 1$, then $G\in \mathcal{G}_{1} = \{S_n\ (n\geq 2),~T(n_1,n_2)(n_1,n_2\geq 1)\}$.
\begin{figure}[H]
\begin{center}
\includegraphics[scale = 0.9]{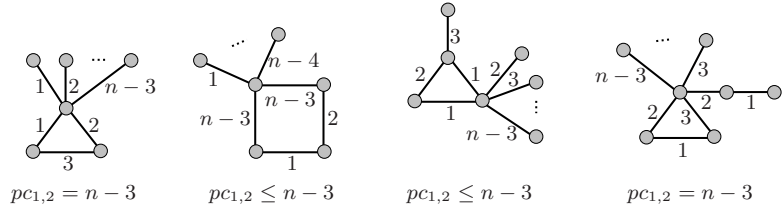}
\caption{Graphs obtained by adding an edge to $T(1,n_2\geq 3)$.}\label{fig6}
\end{center}
\end{figure}

\begin{figure}[H]
\begin{center}
\includegraphics[scale = 0.9]{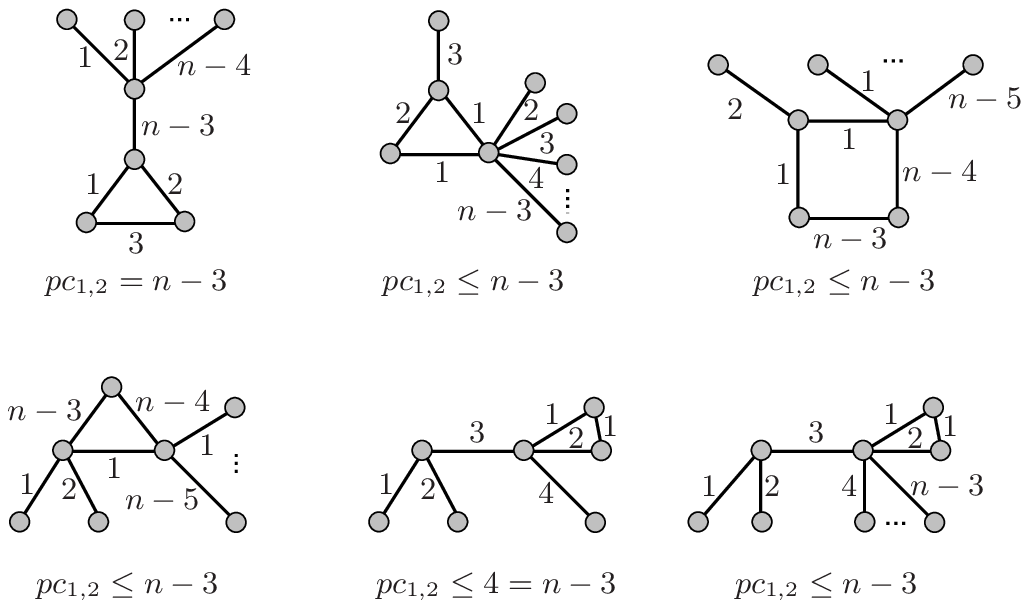}
\caption{Graphs obtained by adding an edge to $T(2,n_2\geq 3)$.}\label{fig7}
\end{center}
\end{figure}

\begin{figure}[H]
\begin{center}
\includegraphics[scale = 0.9]{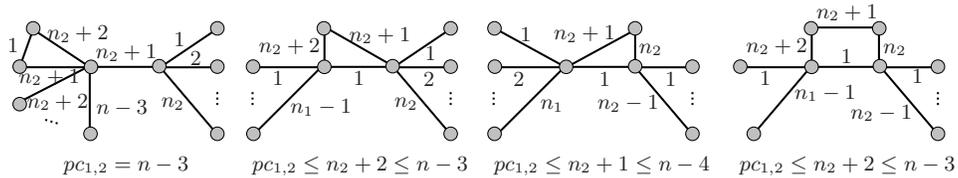}
\caption{Graphs obtained by adding an edge to $T(n_1\geq 3,n_2\geq 3)$.}\label{fig8}
\end{center}
\end{figure}

Proof of $(\romannumeral 2)$: One can easily check that $pc_{1,2}(G)=n-2$ for any graph $G\in \mathcal{G}_2$. Hence, it remains to show the converse. Since $pc_{1,2}(G)=n-2$, then $n-2\leq pc_{1,2}(T)\leq n-1$. Thus, Theorem~\ref{tree} implies that any spanning tree $T$ of $G$ must be an element of the set $\{S_n\ (n\geq 2),~T(n_1,n_2)\ (n_1,n_2\geq 1),~T^1(n_1,n_2)\ (n_1,n_2\geq 1)\}$.

If $G$ is a tree, then $G\cong T^1(n_1,n_2)\ (n_1,n_2\geq 1) \subseteq \mathcal{G}_{2}$. Next we suppose that $G$ is not a tree. Then $G$ can be constructed from $S_n\ (n\geq 2)$, $T(n_1,n_2)\ (n_1,n_2\geq 1)$ or $T^1(n_1,n_2)\ (n_1,n_2\geq 1)$ by adding edges. In the proof of $(\romannumeral 1)$, we listed eight graphs with $(1,2)$-proper connection number $n-2$, which are $C_3,~C_4,~G_1,~G_3,~G_4,~G_6,~G_7$ and $G_8$, respectively. Furthermore, all graphs obtained by adding an edge to $S_n\ (n\geq 2)$ or $T(n_1,n_2)\ (n_1,n_2\geq 1)$ except these eight ones have $(1,2)$-proper connection number no more than $n-3$. Therefore, the graph $G$ can be constructed from $C_3,~C_4,~G_1,~G_3,~G_4,~G_6,~G_7,~G_8$ or $T^1(n_1,n_2)\ (n_1,n_2\geq 1)$ by adding edges.
\begin{figure}[H]
\begin{center}
\includegraphics[scale = 0.9]{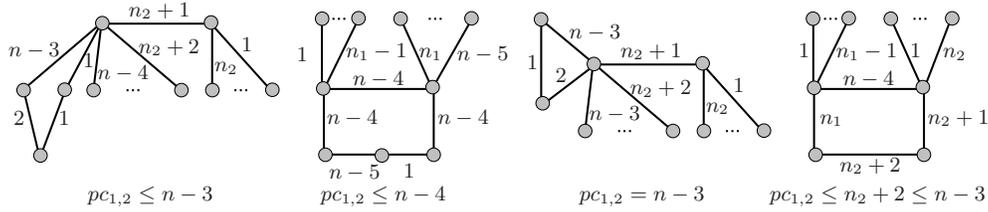}
\caption{Graphs obtained by adding an edge to $T^1(n_1\geq 2,n_2\geq 2)$.}\label{fig9}
\end{center}
\end{figure}

Considering graphs constructed from $C_3,~C_4,~G_1,~G_3,~G_4,~G_6,~G_7$ or $G_8$ by adding edges, we find only another two graphs $G_2,~G_5$ with $pc_{1,2}(G_2)=2=|V(G_2)|-2$ and $pc_{1,2}(G_5)=3=|V(G_5)|-2$. All others have $(1, 2)$-proper connection number no more than $n - 3$. Now we focus on the graphs obtained by adding an edge to $T^1(n_1,n_2)\ (n_1,n_2\geq 1)$. For the cases $n_1=n_2=1$, $n_1=1,~n_2\geq 2$ and $n_1\geq 2,~n_2=1$, we find another graph $C_5$ such that $pc_{1,2}(C_5)=n-2$ with similar analysis as in the proof of $(\romannumeral 1)$. Denote by $e$ the new edge added to $T(n_1,n_2)\ (n_1,n_2\geq 1)$ or $T^1(n_1,n_2)\ (n_1,n_2\geq 1)$ and $T(n_1,n_2)+e$, $T^1(n_1,n_2)+e$ the newly obtained graphs. For the case $n_1\geq 2,~n_2\geq 2$, we consider cases depending on whether the pendent vertex $u_0$ in $T^1(n_1,n_2)$ is an end vertex of $e$ or not. It is obvious that if $u_0\notin e$, then $T^1(n_1,n_2)+e\setminus u_0\cong T(n_1,n_2)+e$. The proof of $(\romannumeral 1)$ suggests that we only need to consider the case when $T^1(n_1,n_2)+e\setminus u_0\cong G_8$. It is easy to check that $pc_{1,2}(T^1(n_1,n_2)+e)=n-3<n-2$ for this case. If $u_0\in e$, then one of the graphs in Figure~\ref{fig9} will be obtained by adding an edge to $T^1(n_1,n_2)$. However, all these graphs have $(1,2)$-proper connection number no more than $n-3$ (as colored in the figure). Thus, we complete the proof of $(\romannumeral 2)$.
\end{proof}

\begin{thm}\label{Nordhaus-Guddum}
Let $G$ and $\overline{G}$ be connected graphs on $n$ vertices. Then $pc_{1,2}(G)+pc_{1,2}(\overline{G})\leq n+2$ and the equality holds if and only if $G$ or $\overline{G}$ is isomorphic to a double star, i.e., $G\cong T(n_1,n_2)\ (n_1,n_2\geq 1)$ or $\overline{G}\cong T(n_1,n_2)\ (n_1,n_2\geq 1)$.
\end{thm}

\begin{proof}
Since both $G$ and $\overline{G}$ are connected, we have $n\geq 4$ and $\Delta(G),~\Delta(\overline{G})\leq n-2$. Let $G$ be the double star with center vertices $u,v$ and $N_G(u)\setminus v=A,~N_G(v)\setminus u=B$. So, $\overline{G}[A\cup B]$ is a clique and $N_{\overline{G}}(u)=B,~N_{\overline{G}}(v)=A$. Certainly all edges of $G$ must have distinct colors so we consider colorings of $\overline{G}$. Color all edges incident to $v$ with $1$, all edges incident to $u$ with $2$ and edges in $\overline{G}[A\cup B]$ with $3$. This coloring shows that $pc_{1,2}(\overline{G}) \leq 3$. Since $u$ and $v$ are at distance $3$ in $\overline{G}$, we get that $pc_{1, 2}(\overline{G}) = 3$ and so $pc_{1,2}(G)+pc_{1,2}(\overline{G})=n+2$. Now, we must show that $pc_{1,2}(G)+pc_{1,2}(\overline{G})<n+2$ for all other connected graphs $G$ and $\overline{G}$. One can easily check that this is true for $n=4,~5$. So we consider $n\geq 6$ in the following.

If $G$ or $\overline{G}$ has $(1,2)$-proper connection number $n-1$ or $n-2$, i.e., $G\in \mathcal{G}_1\cup \mathcal{G}_2\setminus T(n_1,n_2)\ (n_1,n_2\geq 1)$ or $\overline{G}\in \mathcal{G}_1\cup \mathcal{G}_2\setminus T(n_1,n_2)\ (n_1,n_2\geq 1)$, then $pc_{1,2}(G)+pc_{1,2}(\overline{G})<n+2$ by simple examination. Hence, we can assume that $2\leq pc_{1,2}(G)\leq n-3$ and $2\leq pc_{1,2}(\overline{G})\leq n-3$.

Suppose first that both $G$ and $\overline{G}$ are $2$-connected. For $n=6$, it is easy to check that $pc_{1,2}(G)+pc_{1,2}(\overline{G})\leq 3+3<8=n+2$. And for $n\geq 9$, Theorem \ref{$2$-connected} implies that $pc_{1,2}(G)+pc_{1,2}(\overline{G})\leq 5+5=10<11\leq n+2$. Then what remains are the cases $n=7$ and $n=8$. For convenience, we denote the circumference of $G$ by $c(G)$. We first suppose $n=7$. Obviously $4\leq c(G)\leq 7$. If $c(G)=7$, then $C_7$ is a spanning subgraph of $G$ and $pc_{1,2}(G)\leq pc_{1,2}(C_7)=3$. If $c(G)=6$, then $G$ has a traceable spanning subgraph which is composed of $C_6$ by adding an open ear of length two. Thus, $pc_{1,2}(G)\leq 3$. If $c(G)=5$, then $G$ contains $H^7_1$ or $H^7_2$ (see Figure \ref{fig10}) as a spanning subgraph. Since $H^7_1$ is traceable and $pc_{1,2}(H^7_2)\leq 3$, then $pc_{1,2}(G)\leq 3$. For the case $c(G)=4$, $G$ contains $K_{2,5}$ as its spanning subgraph, which contradicts the assumption that $\overline{G}$ is connected. Therefore, all $2$-connected graphs of order $n=7$ with connected complementary graphs has $(1,2)$-proper connection number no more than $3$. Hence, $pc_{1,2}(G)+pc_{1,2}(\overline{G})\leq 3+3<9=n+2$. With similar analysis as for the situation $n=7$, we can also draw the conclusion that $pc_{1,2}(G)+pc_{1,2}(\overline{G})\leq 3+3<10=n+2$ for $n=8$.
\begin{figure}[H]
\begin{center}
\includegraphics[scale = 0.9]{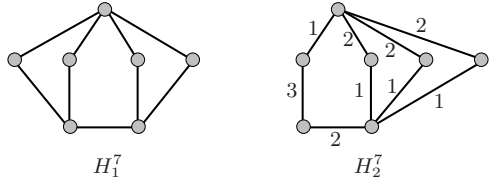}
\caption{Graphs for the proof of Theorem \ref{Nordhaus-Guddum}.}\label{fig10}
\end{center}
\end{figure}

Now we consider the case where at least one of $G$ and $\overline{G}$ has at least one cut vertex. Without loss of generality, suppose that $G$ has at least one cut vertex. We distinguish the following two cases.

\textbf{Case 1:} $G$ has a cut vertex $u$ such that $G-u$ has at least three components.

Let $G_1,~G_2,\cdots,G_k~(k\geq 3)$ be the components of $G-u$, and let $n_i$ be the number of vertices of $G_i$ for $i=1,2,\dots,k$ with $n_1\leq n_2\leq \cdots\leq n_k$. Since $\Delta(G)\leq n-2$, then $n_k\geq 2$. The complementary graph $\overline{G}\setminus u$ contains $K_{n_k,n-n_k-1}$ as a spanning subgraph and both $n_k \geq 2$ and $n-n_k-1\geq 2$. By Theorem~\ref{bipartite}, there exists a $(1,2)$-proper-path $3$-coloring of $K_{n_k,n-n_k-1}$ using elements in $[3]$. Then, if we color the edges incident to $u$ in $\overline{G}$ with color $4$, then we obtain a $(1,2)$-proper-path $4$-coloring of $\overline{G}$. Therefore, $pc_{1,2}(G)+pc_{1,2}(\overline{G})\leq (n-3)+4=n+1<n+2$.

\textbf{Case 2:} Each cut vertex $u$ of $G$ satisfies that $G-u$ has only two components.

Let $G_1,~G_2$ be the two components of $G-u$, and let $n_i$ be the number of vertices of $G_i$ for $i=1,2$ with $n_1\leq n_2$.  Since $n\geq6$, then $n_2\geq 2$.

\textbf{Subcase 2.1:} $n_1\geq 2$.
The complementary graph $\overline{G}\setminus u$ contains $K_{n_1,n_2}$ as a spanning subgraph. By Theorem~\ref{bipartite}, there is a coloring of $K_{n_{1}, n_{2}}$ with colors in $[3]$, and we color the edges incident to $u$ in $\overline{G}$ with color $4$. This gives a $(1,2)$-proper-path $4$-coloring of $\overline{G}$. As a result, $pc_{1,2}(G)+pc_{1,2}(\overline{G})\leq n-3+4=n+1<n+2$ as desired.

\textbf{Subcase 2.2:} $n_1=1$, i.e., each cut vertex of $G$ is incident with a pendent edge.

Since $n\geq 6$, then $n_2\geq 4$. Let $\{u_1,~u_2,\dots,u_{\ell}\}$ be the set of all cut vertices of $G$, and let $u_1v_1,~u_2v_2,\dots,u_{\ell}v_{\ell}$ be the pendent edges incident to these cut vertices in $G$. Set $H=G\setminus\{v_1,~v_2,\dots,v_{\ell}\}$, so $H$ is $2$-connected. By Theorem~\ref{$2$-connected}, we know that $pc_{1,2}(H)\leq 5$.
\begin{figure}[H]
\begin{center}
\includegraphics[scale = 0.9]{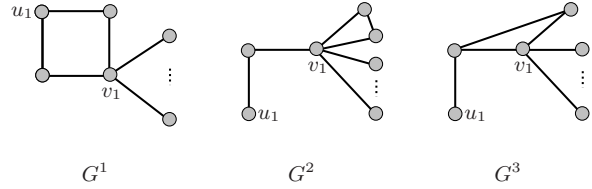}
\caption{Graphs for the proof of Theorem \ref{Nordhaus-Guddum}.}\label{fig11}
\end{center}
\end{figure}

If $\ell\geq 2$, then $\overline{G}\setminus\{u_1,u_2\}$ contains $K_{2,n-4}$ as a spanning subgraph. By Theorem~\ref{bipartite}, there is a coloring of $K_{2,n-4}$ using colors from $[3]$, and we color the edges incident to $u_1$ or $u_2$ in $\overline{G}$ with color $4$. One can easily check this is a $(1,2)$-proper-path $4$-coloring of $\overline{G}$. Thus, $pc_{1,2}(G)+pc_{1,2}(\overline{G})\leq (n-3)+4=n+1<n+2$.

Thus, we may assume $\ell=1$, so $pc_{1,2}(G)\leq pc_{1,2}(H)+1\leq 6$. Since $\overline{G}$ is connected, then $|N_{\overline{G}}(u_1)|\geq 1$ and $\overline{G}$ contains $G^1$, $G^2$ or $G^3$ (see Figure~\ref{fig11}) as a spanning subgraph. We first suppose that $G^1$ is a spanning subgraph of $\overline{G}$. Let $H_{1}, \dots, H_{5}$ be as in Figure~\ref{fig12}. If $\overline{G}\cong H_1$, then it is easy to verify that $pc_{1,2}(G)+pc_{1,2}(\overline{G})=3+3=6<8=n+2$ for $n=6$ and $pc_{1,2}(G)+pc_{1,2}(\overline{G})=4+3=7<9=n+2$ for $n=7$. If $\overline{G}\cong H_1$ and $n\geq 8$, the coloring depicted in Figure \ref{fig12} shows that $pc_{1,2}(\overline{G}) \leq n-4$. In addition, if we color $u_1v_1$ with color $1$, other edges incident to $u_1$ with color $2$ and all other edges color $3$ in $G$, then we get a $(1,2)$-proper-path $3$-coloring of $G$. Consequently, $pc_{1,2}(G)+pc_{1,2}(\overline{G})\leq (n-4)+3=n-1<n+2$. Next we consider the situation $H_1\varsubsetneqq \overline{G}$. Adding an edge to $G^1$, we arrive at some graph in $\{H_2,~H_3,~H_4,~H_5\}$ depicted in Figure~\ref{fig12}. If $\overline{G}\cong H_5$, then $pc_{1,2}(\overline{G})\leq n-4$ by the coloring in Figure~\ref{fig12}. In order to color $G$, we color $u_1v_1$ with color $1$ and other edges incident to $u_1$ with color $2$. Additionally, we color edges incident to $x$ ($y$ is the same) with colors $1,~3$ such that both $1$ and $3$ appear and all other edges with color $2$ in $G$. Thus, we get a $(1,2)$-proper-path $3$-coloring of $G$ and so $pc_{1,2}(G)+pc_{1,2}(\overline{G})\leq 3+(n - 4)=n-1<n+2$. If $\overline{G}$ is not isomorphic to $H_5$, then $\overline{G}$ has $H_2,~H_3$ or $H_4$ as its spanning subgraph. As is depicted in Figure \ref{fig12}, $pc_{1,2}(H_i)\leq n-5~(2\leq i\leq 4)$ for $n\geq 9$. Therefore, $pc_{1,2}(G)+pc_{1,2}(\overline{G})\leq 6+(n - 5)=n+1<n+2$ for $n\geq 9$. For the situation $6\leq n\leq 8$, we can verify the result depending on the circumference of $H=G\setminus u_1$ similarly as above. Hence, if $G^1$ is a spanning subgraph of $G$, then $pc_{1,2}(G)+pc_{1,2}(\overline{G})<n+2$. By the same method, we can draw the same conclusion for $G^2$ or $G^3$ as a spanning subgraph of $G$. Therefore, we complete the proof.
\begin{figure}[H]
\begin{center}
\includegraphics[scale = 0.9]{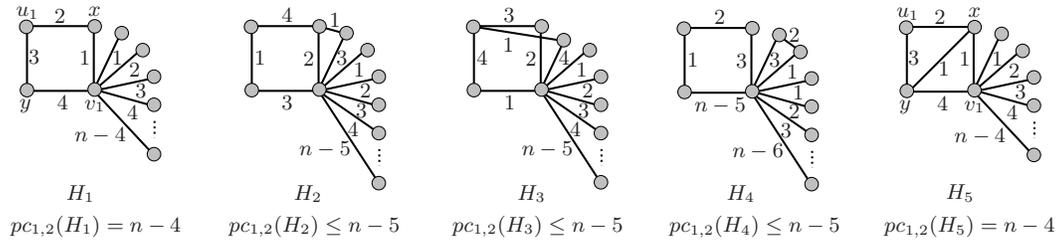}
\caption{Graphs for the proof of Theorem \ref{Nordhaus-Guddum}.}\label{fig12}
\end{center}
\end{figure}
\end{proof}

\end{document}